\documentclass{amsart}
\usepackage{amsfonts}
\usepackage{amsmath}
\usepackage{amsthm}
\usepackage{amssymb}
\usepackage{tikz-cd}
\usepackage{mathrsfs} 
\usepackage{enumerate}
\usepackage{mathtools}
\usepackage{bbm}
\usepackage{graphicx} % Required for inserting images
\usepackage{cleveref}

\newtheorem{theorem}{Theorem}
\newtheorem{lemma}[theorem]{Lemma}
\newtheorem{corollary}[theorem]{Corollary}
\newtheorem{proposition}[theorem]{Proposition}
\newtheorem*{conjecture}{Conjecture}
\newtheorem*{remark}{Remark}
\crefname{theorem}{Theorem}{Theorems}
\crefname{lemma}{Lemma}{Lemmas}
\crefname{corollary}{Corollary}{Corollaries}
\crefname{proposition}{Proposition}{Propositions}
\crefname{conjecture}{Conjecture}{Conjectures}
\theoremstyle{definition}
\newtheorem*{definition}{Definition}

\numberwithin{theorem}{section}

\newcommand{\IC}{\mathrm{IC}}
\newcommand{\Eu}{\mathrm{Eu}}
\newcommand{\rk}{\mathrm{rk}}

\title{Microlocal multiplicity of matroid Schubert varieties}
\author{Yiyu Wang}
\address{Department of Mathematics, University of Wisconsin - Madison, 480 Lincoln Drive, 213 Van Vleck Hall, Madison, WI 53706}
\email{yiyu.wang@wisc.edu}
\date{\today}

\begin{document}

\begin{abstract}
    %We introduce a new invariant for a general matroid. It is the multiplicity number of the characteristic cycle of the intersection complex of the matroid Schubert variety when the matroid is realizable over $\mathbb{C}$. We call it the microlocal multiplicity of a matroid. The microlocal multiplicity is nonnegative in the realizable case and we conjecture it is nonnegative for all matroids.
    We study the multiplicity number of the characteristic cycle of the intersection complex of the matroid Schubert variety. It is shown to be a combinatorial invariant, and it can be computed by explicit formulas. We also conjecture that the generalization to an arbitrary matroid is nonnegative. 
\end{abstract}

\maketitle

\section{Introduction}
Many statements on matroids are proved first in the realizable case by geometric methods or constructions, and then proved in the non-realizable case by combinatorial means. For example, the log-concavity of the characteristic polynomials \cite{10.4007/annals.2018.188.2.1}, the non-negativity of the coefficients of the Kazhdan-Lusztig polynomial \cite{BRADEN2022108646}\cite{braden2023singularhodgetheorycombinatorial}, and most recently the top-heavy conjecture \cite{braden2023singularhodgetheorycombinatorial}. In this paper, we will introduce a new invariant of a matroid which comes from the singularity of the matroid Schubert variety in the realizable case. This new invariant is automatically nonnegative in the realizable case, and we conjecture that it is also nonnegative in the non-realizable case.

Let us introduce the new invariant. Let $V\subset \mathbb{C}^{n}$ be a $d$-dimensional subspace. Throughout the paper, we assume $V$ is not contained in any coordinate hyperplane $\{x_i=0\}$. %\textcolor{magenta}{I recommend to consider $d$-dimensional subspace in $n$-dimensional ambient space. Here and many places afterwards, you need to assume that $V$ is not contained in any coordinate hyperplane.}
%Todo: Replace d with d and n with n.
The intersection of $V$ with each coordinate hyperplane $H_i:\{x_i=0\}$ defines a hyperplane arrangement $\mathcal{A}$ in $V$. Let $M=M(\mathcal{A})$ be the matroid associated with the ground set $E=\{0,1,2,\cdots,n-1\}$. The closure of $V$ in $(\mathbb{P}^1)^{n}$ is the so-called matroid Schubert variety of $\mathcal{A}$, denoted by $Y_\mathcal{A}$. Consider the intersection complex $\IC_{Y_\mathcal{A}}$ of $Y_\mathcal{A}$ with middle perversity (all the sheaves we consider in the paper are with middle perversity). The characteristic cycle of $\IC_{Y_\mathcal{A}}$ can be written as a nonnegative linear combination of conormal varieties:
\[
    \mathrm{CC}(\IC_{Y_\mathcal{A}})=\sum_{F}m_F T^*_{Y_F}(\mathbb{P}^1)^{n}
\]
where the sum is over all flats $F\in \mathcal{L}(M)$, the lattice of all flats of $M$, $\{Y_F\}$ is the stratification of $Y_\mathcal{A}$ as in \cite[Proof of Theorem 14]{10.4310/ACTA.2017.v218.n2.a2}. %with respect to which $\IC_{Y_\mathcal{A}}$ is constructible.
In particular, $m_M$ is the coefficient of $T^*_{[\infty, \cdots, \infty]}(\mathbb{P}^1)^{n}$ in $\mathrm{CC}(\IC_{Y_\mathcal{A}})$.
%before the conormal variety to the all-infinity point $[\infty,\infty,\cdots,\infty]$ of $(\mathbb{P}^1)^{n}$. 
Since $\IC_{Y_\mathcal{A}}$ is a perverse sheaf, $m_M$ is automatically nonnegative. Our main result is that $m_M$ is a combinatorial invariant and can be computed by an explicit formula.

\begin{theorem}\label{Microlocal_multiplicities}
    For a rank $d$ loopless matroid $M=M(\mathcal{A})$ realized by a hyperplane arrangement $\mathcal{A}$, the multiplicity of $\mathrm{CC}(\IC_{Y_\mathcal{A}})$ at the conormal variety to the all-infinity point is 
    $$m_M=(-1)^{d} \sum_{F\in \mathcal{L}(M)}2^{\rk F}\chi_{M^F}(1/2)P_{M_F}(1),$$ where $\mathcal{L}(M)$ denotes the set of all flats of $M$, $\chi_M$ is the characteristic polynomial of $M$, $P_M$ is the Kazhdan-Lusztig polynomial of $M$, $M^F$ and $M_F$ denotes the localization and the contraction of $M$ at the flat $F$ respectively.
\end{theorem}

The above formula makes sense even when $M$ is not realizable over $\mathbb{C}$. Thus, for an arbitrary matroid $M$, we call $m_M$ the microlocal multiplicity of $M$. It is clearly a valuative invariant of a matroid. We conjecture that it is always nonnegative. 
\begin{conjecture}
    For a loopless matroid $M$, $m_M$ is always nonnegative. That is, 
    \[
        (-1)^{d} \sum_{F\in \mathcal{L}(M)}2^{\rk F}\chi_{M^F}(1/2)P_{M_F}(1)\geq 0.
    \]
    Moreover, $m_M=0$ if and only if $M$ has a Boolean summand, that is, $M$ is a direct sum of a Boolean matroid and another arbitrary matroid.
\end{conjecture}

When $M$ is a direct sum of two matroids $M_1$ and $M_2$, we proved in \cref{productive_m} that $m_{M_1\oplus M_2}=m_{M_1}m_{M_2}$. Therefore, $m_M=0$ if $M$ is a direct sum of a Boolean matroid with an arbitrary matroid. This explains the second sentence in the above conjecture. 

To calculate the microlocal multiplicity, we use the fact that the constructible function of a constructible sheaf defined by taking the stalk-wise Euler characteristic number depends only on the characteristic cycle of the constructible sheaf. The stalk-wise Euler characteristic of the constructible sheaf is exactly given by the Kazhdan-Lusztig polynomial, evaluating at 1. The conormal varieties, on the other hand, correspond to Macpherson's local Euler obstruction functions. Our second result, which is also part of the proof of the first theorem, is the value of the local Euler obstruction of the matroid Schubert variety. Here we only state the value of the local Euler obstruction at the all-infinity point, but the value at other points can be reduced to this case. More precisely, $\Eu$ at a point in $Y_F$ is given by $\Eu_{M_F}$.

\begin{theorem}\label{Non_Recursive_Eu}
    For a rank $d$ loopless matroid $M=M(\mathcal{A})$ realized by a hyperplane arrangement $\mathcal{A}$, the value of the local Euler obstruction at the all-infinity point $\Eu_M$ is equal to $\chi_M(2)$.
\end{theorem}

Combining with the results on CSM classes of matroid Schubert varieties \cite[Theorem 1.15]{Eur_Huh_Larson_2023}, we can write down the Chern-Mather class of the matroid Schubert variety.

\begin{corollary}\label{Closed_CM}
    For a rank $d$ loopless matroid $M=M(\mathcal{A})$ realized by a hyperplane arrangement $\mathcal{A}$, the Chern-Mather class of $Y_\mathcal{A}$ is given by the following formula:
    \[
        c_{Ma}(Y_\mathcal{A})=\sum_{F\in\mathcal{L}(M)}2^{\rk M-\rk F}y_F\in A_*(Y_\mathcal{A}).
    \]
\end{corollary}

A lot of work is done concerning the positivity of the local Euler obstruction in the settings of the ordinary Schubert variety. For example, Mihalcea and Singh proved if $G/P$ is cominuscule then any Schubert subvariety of $G/P$ has nonnegative local Euler obstruction function, see \cite{mihalcea2020mather}. We also examined such positivity here, and to our surprise, the Euler obstruction function of $Y_\mathcal{A}$ is always not everywhere positive unless $M$ is a boolean matroid.

\begin{theorem}\label{non-positivity}
    For a simple rank $d$ matroid $M=M(\mathcal{A})$ realized by a hyperplane arrangement $\mathcal{A}$, the local Euler obstruction function of $Y_\mathcal{A}$ is everywhere strictly positive if and only if $M$ is a Boolean matroid. 
\end{theorem}

%Can you also compute the Chern-Mather classes of the Matroid Schubert variety and the reciprocital plane (in the projective space)? They can be included as the main results too. You can also compare them with the CSM classes.

The paper is organized as follows. Section 2 is devoted to some background material like the Euler obstruction functions and matroid Schubert varieties. We will prove a recursive formula for the Euler obstruction number in section 3. In section 4, we will prove \cref{Non_Recursive_Eu} and \cref{non-positivity}. Finally, in section 5, we will discuss the microlocal multiplicity.

\subsubsection*{Acknowledgements.} The author would like to thank his advisor Botong Wang for various helpful discussions and suggestions, especially the usage of the CSM classes of matroids. He also would like to thank Colin Crowley for his suggestions on the first draft and an improvement on the statement of the conjecture.

\section{Background}

\subsection{The geometry of matroid Schubert varieties}
The geometry of a matroid Schubert variety is very close to an affine toric variety. We will briefly summarize the facts about the matroid Schubert varieties and fix notations in the following paragraphs.

Let $V\subset \mathbb{C}^{n}$ be a linear subspace of dimension $d$. Let $M$ be the loopless matroid on $E=\{0,1,\cdots,n-1\}$ of
$V$. Recall a subset $B\subset E$ is a base of $M$ if and only if $V\to\mathbb{C}^E\twoheadrightarrow \mathbb{C}^B$ is an isomorphism.

 The linear subspace $V$ as an additive group acts on itself and $\mathbb{C}^{n}$ by addition, so this action naturally induces an action on $Y_\mathcal{A}$. Like an affine toric variety, $Y_\mathcal{A}$ has only finitely many $V$-orbits, and these orbits are one-to-one corresponding to the \emph{flats} of $M$. More precisely, for a flat $F$ of $M$, write $\pi_F:\mathbb{C}^{n}\to \mathbb{C}^F$, then the orbit corresponding to $F$ is $\pi_F(V)\times \{\infty\}^{E\backslash F}$. The orbits give a Whitney stratification of $Y_\mathcal{A}$. For more information about matroid Schubert varieties, the interested readers are directed to \cite{crowley2023hyperplane}.

The action of $V$ on $Y_\mathcal{A}$ has a unique fixed point $\{\infty\}^{n}$. It is the orbit corresponding to the empty set, which is a flat. The intersection of $Y_\mathcal{A}$ and the affine chart $(\mathbb{P}^1\setminus \{0\})^{n}$ is called the affine reciprocal plane of the matroid $M$, denoted by $X_\mathcal{A}$. The stratification of $Y_\mathcal{A}$ induces a stratification of $X_\mathcal{A}$, so the strata of $X_\mathcal{A}$ are also one-to-one corresponding to flats. We will use $X_{\mathcal{A},F}$ to denote the stratum that corresponds to the flat $F$. In particular, the big open stratum of $X_\mathcal{A}$ is $V\cap (\mathbb{C}^*)^{n}$ which we will denote by $U_\mathcal{A}$. This is exactly the complement of the hyperplane arrangement $\mathcal{A}$. The following statement is proved in \cite{ELIAS201636}.

\begin{proposition}[{\cite[Proposition 3.1]{ELIAS201636}}]
    The stratum $X_{\mathcal{A},F}$ is isomorphic to $U_{\mathcal{A}^F}$, and its closure in $X_\mathcal{A}$ is isomorphic to $X_{\mathcal{A}^F}$.
\end{proposition}

Here, the notation $\mathcal{A}^F$ and $\mathcal{A}_F$ refers to the localization and the contraction of the hyperplane $\mathcal{A}$. On the level of matroid, we use $M^F$ and $M_F$ to denote the localization and the contraction of $M$ at the flat $F$. See \cite{braden2023singularhodgetheorycombinatorial} for the precise definitions.

The local geometry of the affine reciprocal plane is described by the following theorem which roughly says that $X_{\mathcal{A}_F}$ is an ``\'etale slice'' to the stratum $X_{\mathcal{A},F}$.

\begin{theorem}[{\cite[Theorem 3.3]{ELIAS201636}}] \label{inductive_slice}
    Let $F$ be a flat of $M$ and let $x\in X_{\mathcal{A},F}$. There exists an open subvariety $W$ containing $x$ and a map $\Phi:W\to X_{\mathcal{A}_F}\times X_{\mathcal{A},F}$ such that $\Phi(x)=(0,x)$ and $\Phi$ is \'etale at x. Here $0$ is the original point in the affine chart $(\mathbb{C}^*\cup \infty)^{n}$.
\end{theorem}

\subsection{Constructible sheaves, Constructible functions, Characteristic cycles and Local Euler obstructions}

Let $X$ be a smooth projective variety. After choosing a stratification on $X$, we can define the bounded derived category of $\mathbb{C}$-constructible sheaves $D^b_c(X)$. Given a constructible complex $\mathcal{F}^\bullet$, we get an associated constructible function by taking stalk-wise Euler characteristic:
\[
    p \longrightarrow \chi(\mathcal{F}^\bullet_p).
\]

If we denote $F(X)$ the group of all constructible functions on $X$, one can verify we actually get a group homomorphism $K_0(D^b_c(X))\to F(X)$, where $K_0$ denotes the Grothendieck group.

The characteristic cycle is another thing one can associate with a constructible sheaf. It defines a group homomorphism
\[
    K_0(D^b_c(X))\to L(X),
\]
where $L(X)$ denotes the group of all conic Lagrangian cycles in the cotangent bundle $T^*X$. Any irreducible conic Lagrangian subvariety on the cotangent bundle is a conormal variety to an irreducible subvariety $Z\subset X$, which is denoted by $T^*_Z X$. Recall that the conormal variety to $Z$ is defined to be the closure of $T^*_{Z_{reg}} X$, where $Z_{reg}$ is the smooth part of $Z$. As a conclusion, $L(X)$ is generated by $T^*_Z X$.

It turns out that $L(X)$ and $F(X)$ are naturally isomorphic to each other. The isomorphism is defined via the construction of the Euler obstruction function, see \cite{MacPherson1974}. For each $T^*_Z X$, we map it to $(-1)^{\dim Z}\Eu_Z$, where $\Eu_Z$ is the local Euler obstruction function of $Z$. Moreover, this isomorphism fits into the group homomorphism from $K_0(D^b_c(X))$. More precisely, the following diagram commutes:
\[
% https://tikzcd.yichuanshen.de/#N4Igdg9gJgpgziAXAbVABwnAlgFyxMJZARgBoAGAXVJADcBDAGwFcYkQBpAfXIAoARAHoAjLgGNeADQCU0kAF9S6TLnyEUAJlLFqdJq3YAxKXMXLseAkXLbdDFm0QgAMiYW6YUAObwioAGYAThAAtkg2IDgQSMRmIEGh4TRRSBpxCWGIWpHRiGQgjPTCMIwACiqW6iCBWF4AFjggNPYGTgA6bSH0OHWBIcAAoszyTSB1MPRQSGDMjIw0cHVY-o2IEYxYYI4gcBAbU8n0WIzskFvu8kA
\begin{tikzcd}
                                & K_0(D^b_c(X)) \arrow[rd,"\chi"] \arrow[ld, "\mathrm{CC}"'] &      \\
L(X) \arrow[rr, "\Eu", "\cong"'] &                                     & F(X)
\end{tikzcd}
\]

In particular, it makes sense to speak about the characteristic cycle of a constructible function, which is just the inverse isomorphism of $\Eu$.

\subsection{Chern classes}
Macpherson proved in his remarkable paper \cite{MacPherson1974} that there is a natural transformation $c_*$ from the functor $F$, which assigns an $X$ with $F(X)$, to the functor $H_*$, which assigns an $X$ with its homology group $H_*(X)$, commuting with the proper push forward. Later, $H_*$ is improved to $A_*$, the Chow group functor. In other words, given a smooth projective variety $X$, there is a group homomorphism $c_*$ which assigns a constructible function with an element in $A_*(X)$. There are two important cases. For a subvariety $Z$, the indicator function $1_Z$ is mapped to the so-called Chern-Macpherson-Schwartz class $c_{SM}(Z)$, or CSM class for short; the local Euler obstruction function $\Eu_Z$ is mapped to the so-called Chern-Mather class $c_{Ma}(Z)$. When $Z$ is smooth they all coincide with $c(Z)\cap [X]$, where $c(Z)$ is the ordinary Chern class of $Z$.

Let us end the subsection with Aluffi's result which shows that the Chern-Mather class can be used to calculate the local Euler obstruction.

\begin{proposition}[{\cite[Proposition 3.17]{MR3739192}}]\label{Aluffi_Theorem}
    Let $V$ be a cone over a subvariety $W \subset \mathbb{P}^m$ with vertex $\Lambda$. Let $p \in \Lambda$. Then the local Euler obstruction $\Eu_V (p)$ equals
    \begin{equation*}
        \Eu_V(p)=\sum_{j=0}^{\dim W}(-1)^j c_{Ma}(W)_j
    \end{equation*}
where $c_{Ma}(W)_j$ denotes the degree of the $j$-dimensional component of $c_{Ma}(W)_j$.
\end{proposition}

\subsection{Chern-Schwartz-Macpherson classes of matroids} Recall that for a matroid $M$ that is realized by a $d$-dimensional linear subspace $V$ of $\mathbb{C}^{n}$, the stratum $X_{\mathcal{A},F}$ is isomorphic to $U_{\mathcal{A}^F}$, the complement of the hyperplane arrangement $\mathcal{A}^F$. Quotient everything by $\mathbb{C}^*$, we get that $\pi(X_{\mathcal{A},F})$ is the complement of the hyperplane arrangement $\mathbb{P}(\mathcal{A})$ in $\mathbb{P}(V)\cong\mathbb{CP}^{d-1}$. The CSM classes of such complements are calculated in \cite{MR3999674}. 

\begin{definition}
    Suppose $M$ is a rank $d$ matroid on an $n$ elements set. For $0\leq k\leq d-1$, the $k$-dimensional \emph{Chern-Schwartz-MacPherson (CSM) cycle} $\mathrm{csm}_k(M)$ is the $k$-dimensional skeleton of $\mathcal{B}(M)$ (the Bergman fan of $M$) equipped with weights on its top-dimensional cones. If $M$
    is a loopless matroid, the weight of the cone $\sigma_{\mathscr{F}}$ corresponding to a flag of flats $\mathscr{F}\coloneq\{\emptyset=F_0\subsetneq F_1\subsetneq \cdots \subsetneq F_k \subsetneq F_{k+1}=\{0,1,\cdots,n-1\}\}$ is
    \begin{equation*}
        w(\sigma_{\mathscr{F}})\coloneq (-1)^{d-1-k}\beta(M)[\mathscr{F}]= (-1)^{d-1-k}\prod_{i=0}^k \beta(M|F_{i+1}/F_i),
    \end{equation*}
    where $M|F_{i+1}/F_i=M_{F_i}^{F_{i+1}}$ denotes the minor of $M$ obtained by localizing to $F_{i+1}$ and contracting $F_i$, $\beta$ is the beta invariant of $M$. If $M$ has a loop then the CSM class is empty for all $k$.
\end{definition}

The CSM classes of matroids come from the so-called wonderful compactification of the complement of hyperplane arrangement $\mathbb{P}(U_\mathcal{A})$.

\begin{theorem}[{\cite{MR3999674}}]
    Let $W_\mathcal{A}$ be the wonderful compactification of $\mathbb{P}(U_\mathcal{A})$. Then
    \begin{equation*}
        c_*(\mathbbm{1}_{\mathbb{P}(U_\mathcal{A})})=\sum_{k=0}^{d-1} \mathrm{csm}_k(M) \in A_*(W_\mathcal{A}).
    \end{equation*}
\end{theorem}

\section{Local Euler obstructions of matroid Schubert varieties} 
This section is devoted to a recursive formula for the local Euler obstruction, see \cref{Recursive_Eu} below.

To start, let us fix a $V\subset \mathbb{C}^{n}$, and consider the corresponding matroid Schubert variety $Y_\mathcal{A}$. We stratify $Y_\mathcal{A}$ by $V$-orbits, and throughout this section, all constructible functions are constructible with respect to this stratification.

The affine reciprocal plane is an open subset of the matroid Schubert variety, so the Euler obstructions of $X_\mathcal{A}$ are exactly the Euler obstructions of $Y_\mathcal{A}$. By definition, the Euler obstruction function is a constructible function, so the Euler obstruction function is constant when restricting to each stratum. We will use $\Eu_M(F)$ to denote the value of $\mathrm{Eu}_{X_\mathcal{A}}$ at a point of $X_{\mathcal{A},F}$. By \ref{inductive_slice} and the fact an \'etale morphism keeps the value of $\Eu$ unchanged (An \'etale morphism is a local isomorphism in analytic geometry), we conclude that 
\begin{equation}
    \Eu_M(F)=\Eu_{X_\mathcal{A}}(x)=\Eu_{X_{\mathcal{A}_F}}(0)\Eu_{X_{\mathcal{A},F}}(x)=\Eu_{M_F}(\emptyset).
\end{equation}
Since $X_{\mathcal{A},F}$ is smooth, $\Eu_{X_{\mathcal{A},F}}(x)=1$. Notice that the original point in the affine chart $(\mathbb{C}^*\cup \infty)^{n}$ is exactly the orbit that corresponds to the empty flat $\emptyset$. In other words, the value at the $V$-fixed point $\{\infty\}^{n}$ is the only thing we need to calculate.

There is an equivalent way of defining the affine reciprocal plane. Let $i:(\mathbb{C}^*)^{n}\to (\mathbb{C}^*)^{n}$ be the map of taking inverse of each coordinate: $i(x_0,\cdots,x_n)=(x_0^{-1},\cdots,x_n^{-1})$. For a linear subspace $V\subset \mathbb{C}^{n}$, the affine reciprocal plane of $V$ is the closure of $i(V\cap (\mathbb{C}^*)^{n})$ in $\mathbb{C}^{n}$. The set $V\subset \mathbb{C}^{n}$ is invariant under the scaling action of $\mathbb{C}^*$, and so is $X_\mathcal{A}$. This shows $X_\mathcal{A}$ is the affine cone of a projective variety $\mathbb{P}(X_\mathcal{A})\subset \mathbb{P}^n$, the so-called projective reciprocal plane. We are in a position to apply Aluffi's formula, \cref{Aluffi_Theorem}.

\begin{lemma}\label{first_formula}
    Let $\mathcal{A}$ be a rank $d$ hyperplane arrangement represented by a $d$-dimensional subspace $V$ of $\mathbb{C}^{n}$ and let $M$ be the matroid of $\mathcal{A}$. Let $\mathcal{L}(M)$ denote the lattice of flats in $M$. Then we have
    $$\Eu_M(\emptyset)=\sum_{\emptyset\neq F\in \mathcal{L}(M)}\Eu_{M_F}(\emptyset)\left(\sum_{j=0}^{d-1}(-1)^j c_{SM}({\pi(X_{\mathcal{A},F})})_j\right).$$
\end{lemma}
\begin{proof}
    Apply Aluffi's result \cref{Aluffi_Theorem} to the affine reciprocal plane, we get
\begin{equation}\label{Eu_equation}
    \Eu_M(\emptyset)=\sum_{j=0}^{d-1}(-1)^j c_{Ma}(\mathbb{P}(X_\mathcal{A}))_j.
\end{equation}
The projective reciprocal plane has a stratification induced by the projection map $\pi:X_\mathcal{A}-\{\infty\}\to \mathbb{P}(X_\mathcal{A})$. It is the quotient map by $\mathbb{C}^*$-action. Since $\pi$ exhibits $X_\mathcal{A}-\{\infty\}$ as a $\mathbb{C}^*$ bundle over $\mathbb{P}(X_\mathcal{A})$, we conclude that $\Eu_{X_\mathcal{A}}(x)=\Eu_{\mathbb{P}(X_\mathcal{A})}(\pi(x))$ for $x\neq\infty$. Using this fact, we can write $\Eu_{\mathbb{P}(X_\mathcal{A})}$ as
\begin{equation}
    \Eu_{\mathbb{P}(X_\mathcal{A})}=\sum_{\emptyset\neq F\in \mathcal{L}(M)}\Eu_M(F)\mathbbm{1}_{\pi(X_{\mathcal{A},F})},
\end{equation}
where $\mathbbm{1}_{\pi(X_{\mathcal{A},F})}$ is the indicator function of $\pi(X_{\mathcal{A},F})$, the stratum corresponding to $F$ in the projective reciprocal plane. Putting together with \cref{Eu_equation}, we conclude that
\begin{equation*}
    \begin{aligned}
        \Eu_M(\emptyset)&=\sum_{j=0}^{d-1}(-1)^j c_{Ma}(\mathbb{P}(X_\mathcal{A}))_j\\
        &=\sum_{j=0}^{d-1}(-1)^j c_*\left(\sum_{\emptyset\neq F\in \mathcal{L}(M)}\Eu_M(F)\mathbbm{1}_{\pi(X_{\mathcal{A},F})}\right)_j\\
        &=\sum_{\emptyset\neq F\in \mathcal{L}(M)}\Eu_M(F)\left(\sum_{j=0}^{d-1}(-1)^j c_*(\mathbbm{1}_{\pi(X_{\mathcal{A},F})})_j\right)\\
        &=\sum_{\emptyset\neq F\in \mathcal{L}(M)}\Eu_{M_F}(\emptyset)\left(\sum_{j=0}^{d-1}(-1)^j c_{SM}({\pi(X_{\mathcal{A},F})})_j\right)
    \end{aligned}
\end{equation*}
\end{proof}

The notation $c_{SM}({\pi(X_{\mathcal{A},F})})$ refers to the Chern-Schwartz-Macpherson class of $\pi(X_{\mathcal{A},F})$. Once we compute the degrees of these classes, we can inductively compute the local Euler obstruction numbers of a matroid Schubert variety. Because the closure of $X_{\mathcal{A},F}$ is $X_{\mathcal{A}^F}$, the calculation is only necessary when $F=\emptyset$. Recall that $X_{M,\emptyset}$ is $U_\mathcal{A}$, the complement of the hyperplane arrangement $\mathcal{A}$, and $\mathbb{P}(U_\mathcal{A})$ is the complement of the projective hyperplane arrangement.

\begin{lemma}\label{alternating_csm}
    Let 
    \begin{equation*}
        c_M\coloneqq\sum_{j=0}^{d-1}(-1)^j c_{SM}(\mathbb{P}(U_\mathcal{A}))_j,
    \end{equation*}
    then
    \[
        c_M=(-1)^{d-1}\sum_\mathscr{F}\beta(M)[\mathscr{F}],
    \]
    where $\mathscr{F}$ ranges over all descending flags of flats. For a flag of flats $\emptyset\subset F_1 \subset \cdots \subset F_l\subset M$, it is descending if $\min(F_1)>\min(F_2)>\cdots>\min(F_l)>0$, and the quantity $\beta(M)[\mathscr{F}]=\beta(M^{F_1})\beta(M_{F_1}^{F_2})\cdots\beta(M_{F_l})$ is a product of beta invariants.
\end{lemma}

The proof of the lemma uses the wonderful compactification $W_\mathcal{A}$ of the projective hyperplane arrangement $\mathbb{P}(U_\mathcal{A})$. It was first defined in \cite{Concini1995WonderfulMO}.
\begin{proof}[Proof of \cref{alternating_csm}]
    There is a birational map $p$ from $W_\mathcal{A}$ to $\mathbb{P}(X_\mathcal{A})$, extended from the inclusion map $\mathbb{P}(U_\mathcal{A})\to \mathbb{P}(X_\mathcal{A})$. Since $c_*$ commutes with proper pushforward, the class $c_*(\mathbbm{1}_{\mathbb{P}(U_\mathcal{A})})\in A_*(\mathbb{P}^n)$ is the pushforward of $c_*(\mathbbm{1}_{\mathbb{P}(U_\mathcal{A})})\in A_*(W_\mathcal{A})$. Let $h\in A^1(\mathbb{P}^n)$ be the hyperplane class of $\mathbb{P}^n$. The sum $$\sum_{j=0}^{d-1}(-1)^j c_{SM}(\mathbb{P}(U_\mathcal{A}))_j$$ can be rewritten as $$\sum_{j=0}^{d-1}(-1)^j c_{SM}(\mathbb{P}(U_\mathcal{A}))\cap h^j$$ in $A_0(\mathbb{P}^n)\cong \mathbb{Z}$. By the projection formula, we have
\begin{equation*}
    c_{SM}(\mathbb{P}(U_\mathcal{A}))\cap h^j=p_*(c_*(\mathbbm{1}_{\mathbb{P}(U_\mathcal{A})}))\cap h^j=p_*(c_*(\mathbbm{1}_{\mathbb{P}(U_\mathcal{A})})\cap (p^*h)^j)
\end{equation*}
The pushforward $p_*$ is the identity map on $A_0$ because $p$ is birational. It remains to calculate the element $$c_*(\mathbbm{1}_{\mathbb{P}(U_\mathcal{A})})\cap (p^*h)^j\in A_0(W_\mathcal{A})\cong \mathbb{Z}.$$

The hyperplane class $p^*h\in A^1(W_\mathcal{A})$ was studied in the papers \cite{10.4007/annals.2018.188.2.1} and \cite{BRADEN2022108646}. We will adopt their notations. It is the beta class of $M$
\begin{equation*}
    \beta_M=\sum_{i\notin G}x_G\in A^1(W_\mathcal{A}),
\end{equation*}
where the sum is over all nonempty proper flats $G$ of $M$ not containing a given element $i$ in $E$. The definition does not depend on the choice of $i$. Moreover, there is a combinatorial formula for $\beta_M^j$:
\begin{equation*}
    \beta_M^j=\sum_{\mathscr{F}}x_{\mathscr{F}}
\end{equation*}
where the sum is over all descending $j$-step flags of nonempty proper flags of $M$, and $x_{\mathscr{F}}=x_{F_1}x_{F_2}\cdots x_{F_j}$ where $\mathscr{F}:F_1\subset F_2\subset\cdots\subset F_j$. Recall that a flag of flats is said to be descending if $\min(F_1)>\min(F_2)>\cdots>\min(F_j)>0$. Putting together all these facts, we conclude that
\begin{equation*}
    \begin{aligned}
        c_*(\mathbbm{1}_{\mathbb{P}(U_\mathcal{A})})\cap (p^*h)^j&=\mathrm{csm}_j(M)\cap \beta_M^j\\&=\sum_{\mathscr{F}}\mathrm{csm}_j(M)\cap x_{\mathscr{F}}\\&=\sum_{\mathscr{F}}\mathrm{csm}_j(M)(\sigma_{\mathscr{F}})\\&=\sum_{\mathscr{F}}(-1)^{d-1-j}\beta(M)[\mathscr{F}],
    \end{aligned}
\end{equation*}
and the alternating sum of degrees of the CSM classes is
\begin{equation*}
    \begin{aligned}
        \sum_{j=0}^{d-1}(-1)^j c_{SM}(\mathbb{P}(U_\mathcal{A}))_j&=\sum_{j=0}^{d-1}(-1)^j \sum_{\mathscr{F}}(-1)^{d-1-j}\beta(M)[\mathscr{F}]\\&=(-1)^{d-1}\sum_{j=0}^{d-1} \sum_{\mathscr{F}}\beta(M)[\mathscr{F}],
    \end{aligned}
\end{equation*}
where the internal sum is over all $j$-step descending flags, and $j$ ranges over $0$ to $d$. As a conclusion, the alternating sum on the CSM classes is
\begin{equation*}
    \sum_{j=0}^{d-1}(-1)^j c_{SM}(\mathbb{P}(U_\mathcal{A}))_j=(-1)^{d-1} \sum_{\mathscr{F}} \beta(M)[\mathscr{F}]=c_M,
\end{equation*}
where the sum is over all descending flags of flats, from $0$-step ($\emptyset\subset M$) to $(d-1)$-step ($\emptyset\subset F_1\subset\cdots\subset F_{d-1}\subset M$). 
\end{proof}

Combining two lemmas, we obtain a recursive formula for $\Eu_M$:
\begin{proposition}\label{Recursive_Eu}
    For a loopless matroid $M$, $\Eu_M$ can be recursively calculated by the following formula:
    \[
        \Eu_M=\sum_{F\neq\emptyset}c_{M^F}\Eu_{M_F},
    \]
    and $\Eu_\emptyset=1$.
\end{proposition}

We will solve this recursion in the next section. We end this section with an inductive formula for $c_M$'s.
\begin{lemma}\label{Recursive_c}
    The constant $c_M$ is determined by the following recursive formula:
    \begin{equation}\label{recursive_formula_for_c}
        c_M=(-1)^{d-1} \beta(M)+\sum_{0\notin F\neq \emptyset}(-1)^{\mathrm{cork}F} c_{M^F}\beta(M_F)
    \end{equation}
    and
    \begin{equation*}
        c_\emptyset=-1.
    \end{equation*}
\end{lemma}
\begin{proof}
    By definition,
    \[
        c_M= (-1)^{d-1} \sum_{\mathscr{F}} \beta(M)[\mathscr{F}],
    \]
    where $\mathscr{F}$ ranges over all descending flags of flats. The 0-step flag of flats is just $\emptyset\subset M$, therefore:
    \begin{equation*}
        \begin{aligned}
            c_M&= (-1)^{d-1} \beta(M)+(-1)^{d-1}\sum_{k=1}^{d-1}\sum_{\emptyset \subset F_1\subset\cdots\subset F_k\subset M}\beta(M^{F_1})\beta(M_{F_1}^{F_2})\cdots\beta(M_{F_k})\\
            &=(-1)^{d-1} \beta(M)+(-1)^{d-1}\sum_{k=1}^{d-1}\sum_{0\notin F_k\neq\emptyset}\beta(M_{F_k})\sum_{\emptyset\subset F_1\subset\cdots\subset F_k}\beta(M^{F_1})\beta(M_{F_1}^{F_2})\cdots\beta(M_{F_{k-1}}^{F_k})\\
            &=(-1)^{d-1} \beta(M)+(-1)^{d-1} \sum_{0\notin F\neq\emptyset} \beta(M_F)(-1)^{\mathrm{rk}(F)-1}c_{M^F}\\
            &=(-1)^{d-1} \beta(M)+ \sum_{0\notin F\neq\emptyset}(-1)^{\mathrm{cork}(F)} \beta(M_F)c_{M^F}.
        \end{aligned}
    \end{equation*}

    Notice that the flag $\emptyset\subset F_1\subset\cdots\subset F_k$ is a descending flag of flats in the localization $M^{F_k}$.
\end{proof}

\section{Solving for the Euler obstruction number}

To solve the relation in \cref{Recursive_Eu}, we first have to solve for the constants $c_M$. The following identity is the key fact.

\begin{lemma}
    For any non-empty loopless matroid $M$ and any element $0\in E$, the following identity holds:
    \begin{equation}
        \overline\chi_M(t)=\sum_{0\notin F}\chi_{M^F}(t)(-t)^{\mathrm{rk}M_F-1}\beta(M_F)
    \end{equation}
    where $F$ ranges over all flats that do not contain $0$. 
\end{lemma}
Throughout this section, we make the convention that the empty flat $F=\emptyset$ satisfies $0\notin F$.
\begin{proof}
    The proof is simply "writing everything out". By definition,
    \[
        \chi_{M^F}(t)=\sum_{G\subset F}\mu(\emptyset,G)t^{\rk F-\rk G},
    \]
    and
    \[
        \beta(M_F)=(-1)^{\rk M_F}\sum_{F\subset H}\mu(F,H)(\rk(H)-\rk(F)),
    \]
    so the right-hand side is
    \[
        \sum_{\substack{G\subset F\subset H\\0\notin F}}\mu(\emptyset,G)\mu(F,H)(\rk H-\rk F)t^{\rk F-\rk G}(-t)^{\rk M-\rk F-1}(-1)^{\rk M-\rk F}
    \]
    which simplifies to
    \[
        -\sum_{\substack{G\subset F\subset H\\0\notin F}}\mu(\emptyset,G)\mu(F,H)(\rk H-\rk F)t^{\rk M-\rk G-1}.
    \]
    Since
    \[
        \sum_{H:F\subset H}\mu(F,H)=0
    \]
    for $F\neq M$, we can further simplify the above summation to
    \[
        -\sum_{\substack{G\subset F\subset H\\0\notin F}}\mu(\emptyset,G)\mu(F,H)\rk Ht^{\rk M-\rk G-1}.
    \]
    We rewrite the above summation to be $-\sum_{G\subset F\subset H}+\sum_{\substack{G\subset F\subset H\\0\in F}}$. Let us consider the first sum. By the fact (see for example \cite[Proposition 3.7.2]{Stanley_2011})
    \[
        \sum_{F:G\subset F\subset H}\mu(F,H)=\left\{
        \begin{aligned}
            1,&\quad G=H,\\
            0,&\quad G\neq H.
        \end{aligned}
        \right.,
    \]
    we get
    \[
        -\sum_{G\subset F\subset H}\mu(\emptyset,G)\mu(F,H)\rk Ht^{\rk M-\rk G-1}=-\sum_{G}\mu(\emptyset,G)\rk G t^{\rk M-\rk G-1}.
    \]
    For the second sum, notice that ``$0\in F$ and $G\subset F$'' is equivalent to $G\cup \{0\}\subset F$, which is also equivalent to $\overline{G\cup\{0\}}\subset F$, where the overline means taking closure. Therefore the second sum is equal to
    \[
        \begin{aligned}
            &\sum_{G\subset H}\mu(\emptyset,G)\rk Ht^{\rk M-\rk G-1}\left(\sum_{\overline{G\cup\{0\}}\subset F\subset H}\mu(F,H)\right)\\
            =&\sum_G\mu(\emptyset,G)\rk(\overline{G}\cup\{0\})t^{\rk M-\rk G-1}\\
            =&\sum_{0\in G}\mu(\emptyset,G)\rk Gt^{\rk M-\rk G-1}+\sum_{0\notin G}\mu(\emptyset,G)(\rk G+1)t^{\rk M-\rk G-1}\\
            =&\sum_G \mu(\emptyset,G)\rk Gt^{\rk M-\rk G-1}+\sum_{0\notin G}\mu(\emptyset,G)t^{\rk M-\rk G-1}.
        \end{aligned}
    \]
    We conclude that the original sum is equal to
    \[
        \sum_{0\notin G}\mu(\emptyset,G)t^{\rk M-\rk G-1}.
    \]
    But by \cite[Corollary 7.2.7]{White_1987}, this is exactly $\overline\chi_M (t)$ if $M$ is loopless.
\end{proof}

\begin{corollary}\label{Closed_c}
    For a loopless matroid $M$, $c_M=-2^{\rk M}\chi_M(1/2)$. In particular, if $M=M_1\oplus M_2$, then $c_M=-c_{M_1}c_{M_2}$.
\end{corollary}
\begin{proof}
    If $M=\emptyset$, then both sides are equal to $-1$.
    We only need to verify the right-hand side satisfies the inductive formula for $c_M$ (\cref{Recursive_c}). \cref{Recursive_c} can also be written in the form
    \[
        c_M=\sum_{0\notin F}(-1)^{\rk M_F}c_{M^F}\beta(M_F),
    \]
    so we need to prove
    \[
        -2^{\rk M}\chi_M(1/2)=\sum_{0\notin F}(-1)^{\rk M_F}(-2^{\rk M^F}\chi_{M^F}(1/2))\beta(M_F),
    \]
    equivalently,
    \[
    \begin{aligned}
        \chi_M(1/2)&=\sum_{0\notin F}(-1)^{\rk M_F}(2^{-\rk M_F}\chi_{M^F}(1/2))\beta(M_F)\\
        &=\sum_{0\notin F}\chi_{M^F}(1/2)(-1/2)^{\rk M_F}\beta(M_F).
    \end{aligned}
    \]
    Finally, notice that $\overline{\chi}_M(1/2)=\chi_M(1/2)/(1/2-1)=-2\chi_M(1/2)$, we reduce the above equality to the previous theorem where we evaluate $t$ at $1/2$.
    If $M=M_1\oplus M_2$, then since $2^{\rk M}=2^{\rk M_1}2^{\rk M_2}$ and $\chi_M(t)=\chi_{M_1}(t)\chi_{M_2}(t)$, we conclude that $c_M=-c_{M_1}c_{M_2}$.
\end{proof}

\begin{remark}
    Recall in the statement of \cref{alternating_csm}, $c_M$ is defined to be a huge sum of beta invariants, but this lemma states that this huge sum is equal to $-2^{\rk M}\chi_M(1/2)$.
\end{remark}

We can now calculate the value of $\Eu$.

\begin{lemma}
    For any non-empty loopless matroid $M$,
    \[
        \sum_{F\in \mathcal{L}(M)}t^{\rk F}\chi_{M^F}(1/t)\chi_{M_F}(t)=0.
    \]
\end{lemma}
\begin{proof}
    Expanding everything by its definition,
     we get:
     \[
     \begin{aligned}
         &\sum_{G\subset F\subset H}t^{\rk F}\mu(\emptyset,G)(1/t)^{\rk F- \rk G}\mu(F,H)t^{\rk M-\rk H}\\
         =&\sum_{G\subset F\subset G}\mu(\emptyset,G)\mu(F,H)t^{\rk M-\rk H+\rk G}\\
         =&\sum_{G\subset H}\mu(\emptyset,G)t^{\rk M-\rk H+\rk G}(\sum_{F:G\subset F\subset H}\mu(F,H))\\
         =&\sum_G\mu(\emptyset,G)t^{\rk M}=0.
     \end{aligned}
     \]
\end{proof}
\begin{corollary}\label{Closed_Eu}
    For a loopless matroid $M$, $\Eu_M=\chi_M(2)$. In particular, if $M=M_1\oplus M_2$ then $\Eu_{M}=\Eu_{M_1}\Eu_{M_2}$.
\end{corollary}
\begin{proof}
    If $\rk M=0$, then both sides are equal to 1. We only need to verify $\chi_M(2)$ satisfies the recursive definition of $\Eu_M$. Namely, we want to verify
    \[
        \sum_{F\in\mathcal{L}(M)}c_{M^F}\chi_{M_F}(2)=0,
    \]
    equivalently,
    \[
        \sum_{F\in\mathcal{L}(M)}2^{\rk F}\chi_{M^F}(1/2)\chi_{M_F}(2)=0.
    \]
    This is the previous theorem evaluating $t$ at $2$.
    The case $M=M_1\oplus M_2$ immediately follows from the general fact that the characteristic polynomial for matroids is multiplicative.
\end{proof}

Using this closed formula for $\Eu_M$, we can now prove \cref{Closed_CM}.

\begin{proof}[Proof of \cref{Closed_CM}]
    Since the Euler obstruction function $\Eu_{Y_\mathcal{A}}$ is constant on each stratum $Y_F$ of $Y$, we can write $\Eu_{Y_\mathcal{A}}$ as
    \[
        \Eu_{Y_\mathcal{A}}=\sum_{F\in\mathcal{L}(M)}\Eu_{M_F}\mathbbm{1}_{Y_F}
    \]
    where $\mathbbm{1}_{Y_F}$ is the indicator function of $Y_F$. According to \cite[Theorem 1.15]{Eur_Huh_Larson_2023},
    \[
        c_*(\mathbbm{1}_{Y_\mathcal{A}})=c_{SM}(Y_\mathcal{A})=\sum_{G\in\mathcal{L}(M)}y_G.
    \]
    Applying Macpherson's natural transform $c_*$ to $\Eu_{Y_\mathcal{A}}$, we have
    \begin{equation*}
        \begin{aligned}
            c_{Ma}(Y_\mathcal{A})&=c_*(\Eu_{Y_\mathcal{A}})\\
            &=\sum_{G\in\mathcal{L}(M)}\Eu_{M_G}c_*(\mathbbm{1}_{Y_G})\\
            &=\sum_{G\in\mathcal{L}(M)}\chi_{M_G}(2)\sum_{F\in\mathcal{L}(M^G)}y_F\\
            &=\sum_{F\in\mathcal{L}(F)}\left(\sum_{G:G\geq F} \chi_{M_G}(2)\right)y_F\\
            &=\sum_{F\in\mathcal{L}(M)}2^{\rk M-\rk F}y_F.
        \end{aligned}
    \end{equation*}
\end{proof}

As another corollary, we can now prove \cref{non-positivity}.

\begin{proof}[Proof of \cref{non-positivity}]
    Assume the local Euler obstruction function is everywhere positive. This implies that $\Eu_M(F)>0$ for every flat, or equivalently $\chi_{M_F}(2)>0$ for every flat. We must show $M$ is a Boolean matroid after simplification.

    We do induction on $d$. When $d=2$, the simplification of $M$ is always a uniform matroid. Since $\chi_{U_{n,2}}(2)=2-n<0$ if $n>2$, therefore $n=2$ and $M$ is a Boolean matroid. 

    Suppose the theorem is proved for all loopless matroids of rank strictly smaller than $d$. We want to prove the theorem is true in the case the rank is $d$. Without loss of generality we will replace $M$ by its simplification. In particular, all its rank one flats are singletons $\{i\}$ for $i=0,1,\cdots, n-1$. 

    We first claim that $n=d$. Suppose $n>d$. By relabeling elements, we claim that there are $d$ elements $0,1,\cdots,d-1$ such that
    \begin{center}
        Any three of them are not in the same flat of rank two.
    \end{center}
    Choose any element $0\in E$. By the induction hypothesis, the simplification of $M_{\{0\}}$ is isomorphic to a Boolean matroid. In particular, there are $d-1$ rank two flats $F_k$ that contains $0$, and these flats part $E\backslash\{0\}$. More precisely,
    \[
        \bigsqcup_{k=1}^{d-1} F_k\backslash\{0\}=E\backslash\{0\}.
    \]
    Choose an element $k$ from $F_k\backslash\{0\}$. They must be distinct from each other because of the partition property. I claim any three of them cannot be in one flat of rank two. Suppose there is a rank two flat $G$ that contains $i,j,k$. The element $0$ is not in $G$ because of the partition property. Again, by the inductive hypothesis, there is a rank three flat $H$ such that $$H\backslash\{0\}=F_i\backslash\{0\}\sqcup F_j\backslash\{0\}.$$ Notice that $k\notin H$. Take the intersection of $H$ and $G$, we get a new flat that properly contained in $H$ and $G$. Since $G$ is of rank two, this new flat must be of rank one. But all rank one flats are singletons, and this new flat contains $i$ and $j$, a contraction.

    Because $n>d$, there must be an element $e$ in $E$, lies in one of the rank two flats $F_k$. Without loss of generality, let us assume $e\in F_1$. The above arguments can also apply to $e,2,3,\cdots,d-1$, so we get any there elements of them are not in the same flat of rank two. Let us consider the flats that contain $2$. There are $d$ rank two flats that contain $2$, and each of them contains exactly one of $0,1,3,4,\cdots,d-1$. Since $0,e,2,3,\cdots,d-1$ cannot lie in the same rank two flat, $e$ must be in the flat that contains $1,2$. Intersecting this flat with the flat that contains $0,1,e$, we obtain a rank one flat that contains $1,e$, a contradiction.

    We conclude that there are exactly $d$ elements in the ground set $E$. Therefore, $E$ itself is an independent set, so $M$ must be a Boolean matroid.
\end{proof}

\section{Mircolocal multiplicity}
We will prove \cref{Microlocal_multiplicities} in this section.
On a matroid Schubert variety $Y_\mathcal{A}\subset (\mathbb{P}^1)^{n}$, we are interested in the characteristic cycle of the intersection complex $\mathrm{IC}_{Y_\mathcal{A}}$. Again we choose our stratification to be the one that is defined by the group orbits. Therefore, the characteristic cycle of $\mathrm{IC}_{Y_\mathcal{A}}$ is a formal sum of the conic Lagrangian cycles $T^*_{Y_F}(\mathbb{P}^1)^{n}$, where $Y_F$ is the closure of $U^F$ and $U^F$ is the $V$-orbit that corresponds to the flat $F$. We can now write the characteristic cycle of $\IC_{Y_\mathcal{A}}$ in the form
\begin{equation}\label{def_multiplicity}
    \mathrm{CC}(\mathrm{IC}_{Y_\mathcal{A}})=\sum_{F} m_M(F) T^*_{Y_F}(\mathbb{P}^1)^{n}.
\end{equation}

To calculate these coefficients, we use the $\Eu$ isomorphism described in section 2, and translate \cref{def_multiplicity} into an equation of constructible functions. For the left-hand side, recall that for a constructible complex $\mathcal{F}^\bullet$ on an algebraic variety $X$, the associated constructible function is defined by taking pointwise Euler characteristic number:
\begin{equation*}
    p\to \chi(\mathcal{F}^\bullet_p).
\end{equation*}

For the right-hand side, the $\Eu$ isomorphism turns every $T^*_{Y_F}(\mathbb{P}^1)^{n}$ on the other side of \cref{def_multiplicity} into $(-1)^{\dim Y_F}\Eu_{Y_F}$. As a result, we obtain the following equation.

\begin{equation*}
    \chi({\mathrm{IC}_{Y_\mathcal{A},p}})=\sum_{F:p\in Y_F} m_M(F)(-1)^{\dim Y_F} \Eu_{Y_F}(p),
\end{equation*}
here $p$ is any point in $Y_\mathcal{A}$.

By \cite[Lemma 3.8]{ELIAS201636}, the number $\chi((\IC_{Y_\mathcal{A}})_p)$ for a point $p\in U^G$ is 
\begin{equation*}
    \chi({\mathrm{IC}_{Y_\mathcal{A},p}})=(-1)^{\dim Y_\mathcal{A}}P_{M_G}(1)=(-1)^{\mathrm{rk}M}P_{M_G}(1),
\end{equation*}
where $P_M(t)$ refers to the Kazhdan-Lusztig polynomial of $M$. We need the sign $(-1)^{\mathrm{rk}M}$ because we need to shift the complex by the dimension. The condition that $p\in Y_F$ is equivalent to that $F\geq G$. Therefore, we can rewrite the equation as
\begin{equation*}
    (-1)^{\mathrm{rk}M}P_{M_G}(1)=\sum_{F\geq G} m_M(F)(-1)^{\dim Y_F} \Eu_{Y_F}(p).
\end{equation*}
Moreover, $\dim Y_F=\mathrm{rk}(F)$, and $\Eu_{Y_F}(p)=\Eu_{M^F}(G)=\Eu_{M^F_G}$. We obtain a purely combinatorial equation:
\begin{equation}\label{def_multiplicity2}
    (-1)^{\mathrm{rk}M}P_{M_G}(1)=\sum_{F\geq G} m_M(F)(-1)^{\mathrm{rk}F} \Eu_{M^F_G}.
\end{equation}
We can also define the microlocal multiplicity for a general matroid $M$ by the above equation. 
\begin{lemma}
    The multiplicity number $m_M(F)$ only depends on the matroid $M_F$. In fact, $m_M(F)=m_{M_F}(\emptyset)$.
\end{lemma}
\begin{proof}
    Evaluating \cref{def_multiplicity2} with $M=M, G=F$ and $M=M_F,G=\emptyset$, we see that both $m_M(F)$ and $m_{M_F}(\emptyset)$ satisfy the same recursive definition, so they must be equal to each other.
\end{proof}
Following the same convention in this paper, we define $m_M=m_M(\emptyset)$. Now the definition of $m_M$ can be clearly written as the unique solution to the following equation for every matroid $M$:
\begin{equation}\label{def_multiplicity_2}
    (-1)^{\mathrm{rk}M}P_{M}(1)=\sum_{F} m_{M_F}(-1)^{\mathrm{rk}F} \Eu_{M^F}.
\end{equation}
For example, $m_\emptyset=1$ by setting $M=\emptyset$ in this equation. We prove \cref{Microlocal_multiplicities} by solving this linear system. Actually, \cref{Recursive_Eu} shows exactly the ``inverse matrix" of $\Eu_{M^F}$ is $c_{M^F}$.
\begin{proof}[Proof of \cref{Microlocal_multiplicities}]
    Utilizing \cref{def_multiplicity_2}, we have 
    \[
        (-1)^{d-1} \sum_{F}c_{M^F}P_{M_F}(1)=(-1)^{d-1} \sum_{F}c_{M^F}(-1)^{\mathrm{cork}F}\sum_{G:F\subset G}m_{M_G}(-1)^{\mathrm{rk}G-\mathrm{rk}F} \Eu_{M^G_F}.
    \]
    Here we change the order of summation,
    \[
        (-1)^{d-1} \sum_{G}(-1)^{\mathrm{cork}G}m_{M_G}\sum_{F:F\subset G}c_{M^F} \Eu_{M^G_F},
    \]
    we can see the right hand side is very close to \cref{Non_Recursive_Eu}. We only need to treat the case $G=\emptyset$ and $F=\emptyset$ separately. We do it first for $G$ and get
    \[
        (-1)^{d-1} \sum_{G\neq\emptyset}(-1)^{\mathrm{cork}G}m_{M_G}\sum_{F\subset G}c_{M^F} \Eu_{M^G_F}+(-1)^{d-1}(-1)^{\mathrm{rk}M}m_M c_\emptyset.
    \]
    Recall that $d=\mathrm{rk}M$ and $c_\emptyset=-1$. Then we treat the case $F=\emptyset$:
    \[
        m_M+(-1)^{d-1} \sum_{G\neq\emptyset}(-1)^{\mathrm{cork}G}m_{M_G}(c_\emptyset\Eu_{M^G}+\sum_{\emptyset\neq F\subset G}c_{M^F} \Eu_{M^G_F}).
    \]
    Again, $c_\emptyset=-1$. We are now in a position to apply \cref{Non_Recursive_Eu} with $M=M^G$:
    \[
        m_M+(-1)^{d-1} \sum_{G\neq\emptyset}(-1)^{\mathrm{cork}G}m_{M_G}(-\Eu_{M^G}+\Eu_{M^G}),
    \]
    and this is clearly $m_M$. Combine with \cref{Closed_c}, we finish the proof.
\end{proof}

This closed formula for $m_M$ also shows that it is well-behaved under the direct sum.

\begin{corollary}\label{productive_m}
    For any two matroids $M_1$ and $M_2$, $m_{M_1\oplus M_2}=m_{M_1}m_{M_2}$.
\end{corollary}
\begin{proof}
    A flat of $M=M_1\oplus M_2$ is of the form $F_1\cup F_2$, where $F_i\subset M_i$ is a flat, $i=1,2$. We know from \cref{Closed_c} that $c_{M^F}=-c_{(M_1)^{F_1}}c_{(M_2)^{F_2}}$, and from \cite[Proposition 2.7]{ELIAS201636} that $P_{M_F}(t)=P_{(M_1)_{F_1}}(t)P_{(M_2)_{F_2}}(t)$. Putting all these together, we get
    \[
        m_M=(-1)^{\mathrm{rk}{M_1}+\mathrm{rk}{M_2}-1}\sum_{F_1,F_2}-c_{(M_1)^{F_1}}c_{(M_2)^{F_2}}P_{(M_1)_{F_1}}(1)P_{(M_2)_{F_2}}(1)
    \]
    which is exactly the product of 
    \[
        (-1)^{\mathrm{rk}M_i-1} \sum_{F_i}c_{(M_i)^{F_i}}P_{(M_i)_{F_i}}(1),
    \]
    $i=1,2$.
\end{proof}

\begin{remark}
    This corollary is clearer when we take the geometric point of view. If $\mathcal{A}=\mathcal{A}_1\oplus \mathcal{A}_2$, then $Y_{\mathcal{A}}\cong Y_{\mathcal{A}_1}\times Y_{\mathcal{A}_2}$, so
    \[
        \IC_{Y_{\mathcal{A}}}\cong\IC_{Y_{\mathcal{A}_1}}\boxtimes \IC_{Y_{\mathcal{A}_2}}.
    \]
    As a consequence, there is a similar exterior product relation between characteristic cycles, see \cite{Saito_2017}. In particular, the coefficient $m_{M_1\oplus M_2}$ of the conormal variety of the all-infinity point must be the product of $m_{M_i}$, $i=1,2$.
\end{remark}

By \cref{def_multiplicity} and the fact $\IC_{Y_\mathcal{A}}$ is a perverse sheaf, plus the fact that the characteristic cycle of a perverse sheaf is always effective, we conclude that $m_M\geq 0$ when $M$ is realizable over $\mathbb{C}$. However, it is not easy to tell the positivity of the microlocal multiplicity $m$ of a general matroid $M$. Naturally, we propose the following conjecture:
\begin{conjecture}
    For a loopless matroid $M$, $m_M$ is always nonnegative. 
    Moreover, $m_M=0$ if and only if $M$ has a Boolean summand, that is, $M$ is a direct sum of a Boolean matroid and another arbitrary matroid.
\end{conjecture}

Let us end the paper with a corollary about the irreducibility of the characteristic cycle.

\begin{corollary}
    The characteristic cycle of the intersection complex of a matroid Schubert variety of a simple realizable matroid $M=M(\mathcal{A})$ is irreducible if and only if $M$ is a Boolean matroid.
\end{corollary}
\begin{proof}
    If $\mathrm{CC}(\IC_{Y_\mathcal{A}})$ is irreducible, then 
    \[
        \mathrm{CC}(\IC_{Y_\mathcal{A}})=T^*_{Y_\mathcal{A}}(\mathbb{P}^1)^{n}.
    \]
    Translate the equation into an equation of constructible functions, we get
    \[
        (-1)^{\mathrm{rk}M}P_{M_G}(1)=(-1)^{\mathrm{rk}M}m_M \Eu_{Y_\mathcal{A}}(p)
    \]
    for $p\in U^G$. Therefore,
    \[
        \Eu_M(G)=\frac{P_{M_G}(1)}{m_M}>0
    \]
    for every flat $G$. This implies $M$ must be a Boolean matroid by \cref{non-positivity}.
\end{proof}

\bibliographystyle{amsalpha}
\bibliography{main.bib}

\end{document}